\documentclass{amsart}

\usepackage{amsmath,amssymb,verbatim,hyperref,mathrsfs,graphicx}

\newcommand{\T}{\tau} 
\newcommand{\meanv}{\rho} 

\newcommand{\balpha}{\sigma}
\newcommand{\bbeta}{\delta}
\newcommand{\salpha}{\bar \sigma}
\newcommand{\sbeta}{\bar \delta}

\newcommand{\Xk}{\xi}    
\newcommand{\Ak}{\alpha} 
\newcommand{\Bk}{\beta} 
\newcommand{\Ik}{\zeta} 
\newcommand{\sx}{i}   
\newcommand{\sy}{j}   

\newcommand{\valx}{v} 
\newcommand{\valy}{w} 

\newcommand{\new}[1]{{\em #1}\index{#1}}

\newcommand{\R}{\mathbb{R}}
\newcommand{\inte}[1]{\overset{\;_\circ}{#1}}

\newcommand{\fvi}[3]{#1(#2;#3)}     
\newcommand{\fvia}[4]{#1(#2;#3,#4)}     
\newcommand{\fviab}[5]{#1(#2;#3,#4,#5)}     

\newcommand{\set}[2]{\{#1\mid\,#2\}}

\newcommand{\sE}{\mathbb{E}}       

\def\scaled#1#2{{\mathcal S}_{#1}(#2)}
\def\linear#1#2{{\mathcal L}_{#1}(#2)}

\newcommand{\X}{{[n]}}
\newcommand{\RX}{\R^n} 
\newcommand{\A}{A}
\newcommand{\B}{B}
\newcommand{\Ag}{A}
\newcommand{\Bg}{B}
\newcommand{\Am}{{\A}_{\mathrm M}}
\newcommand{\Bm}{{\B}_{\mathrm M}}
\newcommand{\PA}{\Am}
\newcommand{\PB}{\Bm}
\newcommand{\SA}{\mathsf{SA}}

\newcommand{\ki}{{s}} 
\newcommand{\kj}{{l}} 

\newcommand{\proc}[1]{\left( {#1} \right)_{k\ge 0}}              

\newcommand{\MIN}{\text{\sc min}}
\newcommand{\MAX}{\text{\sc max}}

\newcommand{\argmin}[1]{\underset{#1}{\operatorname{argmin}}}

\newcommand{\eigenrad}[1]{r(#1)}
\newcommand{\bonsall}[1]{r_{#1}}
\newcommand{\eigenvalspr}[1]{\hat{r}_{#1}}
\newcommand{\cw}[1]{\operatorname{cw}_{#1}}

\newcommand{\graph}{\mathcal{G}}
\newcommand{\TC}{{\mathcal T}}
\newcommand{\rs}{c} 
\newcommand{\MC}{\mathcal{M}}
\newcommand{\rgrad}[1]{\operatorname{imD}(#1)}
\newcommand{\rec}[1]{\operatorname{rec}(#1)}

\DeclareMathAlphabet{\mathbbold}{U}{bbold}{m}{n}
\newcommand{\un}{\mathbbold{1}}

\theoremstyle{plain}
\newtheorem{theorem}{Theorem}
\newtheorem{coro}[theorem]{Corollary}
\newtheorem{lemma}[theorem]{Lemma}
\newtheorem{prop}[theorem]{Proposition}

\theoremstyle{definition}
\newtheorem{algo}{Algorithm}
\theoremstyle{remark}

\def\pf{\begin{proof}}
\def\endpf{\end{proof}}
\catcode`\@=11
\def\refcounter#1{\protected@edef\@currentlabel
       {\csname the#1\endcsname}
}
\catcode`\@=12
\newcounter{assume}
\def\theassume{{\rm (A\arabic{assume})}}
\def\myitem{\refstepcounter{assume}\item[]\theassume}
\newenvironment{assumption}{\begin{itemize}\myitem}{\end{itemize}}

\title[Mean payoff games with bounded first return times]{Policy iteration for perfect information stochastic mean payoff games with bounded first return times is strongly polynomial}

\author{Marianne Akian}
\address{Marianne Akian,
INRIA Saclay--\^Ile-de-France and CMAP, \'Ecole 
Polytechnique. Address:
CMAP, \'Ecole Polytechnique,
Route de Saclay,
91128 Palaiseau Cedex, France.}
\email{Marianne.Akian@inria.fr}
\author{St\'ephane Gaubert}
\address{St\'ephane Gaubert,
INRIA Saclay--\^Ile-de-France and CMAP, \'Ecole 
Polytechnique. Address: CMAP, \'Ecole Polytechnique,
Route de Saclay,
91128 Palaiseau Cedex, France.}
\email{Stephane.Gaubert@inria.fr}

\keywords{Stochastic games, policy iteration, mean return time, Doeblin state, cone spectral radius}

\begin{document}

\begin{abstract}
Recent results of Ye and Hansen, Miltersen and Zwick show that 
policy iteration for one or two player (perfect information)
zero-sum stochastic games,
restricted to instances with a fixed
discount rate, is strongly polynomial.
We show that policy iteration for
mean-payoff zero-sum stochastic games is also strongly 
polynomial when restricted to instances with bounded first 
mean return time to a given state.
The proof is based on methods of nonlinear Perron-Frobenius 
theory, allowing us to reduce the mean-payoff problem to
a discounted problem with state dependent discount rate.
Our analysis also shows that policy iteration remains strongly 
polynomial for discounted problems in which the discount rate can be
state dependent (and even negative) at certain states, provided that 
 the spectral radii of the nonnegative matrices associated to 
all strategies are  
bounded from above by a fixed constant strictly less than $1$.
\end{abstract}

\maketitle

\section{Introduction}
\subparagraph{Motivation and earlier works}
Policy iteration algorithm is a classical algorithm
to solve discounted Markov decision problems (one player games)
with finite state and actions spaces. A policy is a map
from the set of states to the set of actions, representing
a Markovian decision rule. The algorithm constructs
a sequence of policies such that the associated sequence of values
is strictly decreasing (assuming that the player minimizes her cost function).
Hence, its number of iteration is bounded by the number of policies. 
The method carries over to discounted zero-sum games with
perfect information, still with finite state and action spaces. 
It now makes external iterations in the space of policies of the first player,
and at each step, solves an auxiliary Markov decision problem,
making then internal iterations in the space of policies of the second player.
Again, the first player never selects twice the same policy,
which entails that the algorithm does terminate in a time
which is bounded by the product of the numbers of policies
of both players. This yields an exponential bound
on the execution time, as the number of policies
of one player can be exponential 
in the number of states. However, this general exponential
bound
does not capture the 
experimental efficiency of the algorithm on most applications.

Some recent results shed light on the behavior
of policy iteration as a function of some
particular parameters, such as the discount factor.
Friedmann constructed in~\cite{Friedmann}
an infinite family of 2-player discounted deterministic games with a discount 
factor tending to $1$, showing that the number of
policy iterations can indeed be exponential. 
Fearnley~\cite{fearnley2} and Andersson~\cite{andersson2009}
extended his result to 1-player stochastic games.
However, Ye showed in~\cite{ye_strong} that 
policy iteration solves 1-player discounted games with
a {\em fixed} discount factor $\lambda<1$ 
in \new{strongly polynomial} time
($\lambda$ is not part of the input).
Then, Hansen, Miltersen and Zwick extended this result
in~\cite{hansen2011strategy} to
zero-sum 2-player discounted games with perfect information, and improved 
Ye's bound.
They showed that the number of external iterations of 
the policy iteration algorithm for 2-player games with a fixed discount factor 
$\lambda<1$ is bounded by:
\begin{equation}\label{borne-hansen}
(m+1) (1+\frac{\log(n^2/(1-\lambda))}{-\log(\lambda)})
={\mathcal O} (\frac{m}{1-\lambda} \log \frac{n}{1-\lambda} ),\end{equation}
where $n$ is the number of states, and $m$ is the \new{total number of actions
of both players}, that is the number of triples $(i,a,b)$ 
where $i$ is a state, $a$ is an action of first player, and $b$ is an
action of second player.

\subparagraph{Contribution}
We show that policy iteration still has a strongly polynomial behavior
for a class of mean payoff games, as well as for a more
general class of discounted games. 

As a preliminary step, we show that we can improve
the bound~\eqref{borne-hansen}, in the original situation considered 
in~\cite{hansen2011strategy}. We replace
this bound by the following one (Theorem~\ref{hansen-improved}):
\begin{equation}\label{notre_borne}
\ki_{\max}:= (m_1-n) (1+\lfloor \frac{\log(1-\lambda)}{\log(\lambda)}
\rfloor )= {\mathcal O} (\frac{m_1-n}{1-\lambda} \log \frac{1}{1-\lambda}),
\end{equation}
with $m_1$ the \new{total number of actions of the first player},
that is the number of couples state-action $(i,a)$.
The above new bound is obtained by adapting the technique of Ye and
Hansen, Miltersen and Zwick to nonlinear maps 
which allows us in particular to replace $m$ by $m_1$. Note that
the bound~\eqref{notre_borne} is linear in the size of the input,
for a fixed $\lambda$.

Then, we consider games with state dependent discount factors,
possibly greater than $1$ locally. We establish
a strongly polynomial bound for the number of iterations (Corollary~\ref{coro-omegabar}) which differs from~\eqref{borne-hansen}
and~\eqref{notre_borne} 
in that the discount factor $\lambda$ is now replaced by the maximum
of the spectral radii of all the transition matrices associated
to pairs of policies of both players.
We introduce a natural scaling transformation,
which has the property of leaving invariant the 
combinatorial trace of the policy iteration algorithm. 
This scaling is obtained using 
techniques of non-linear Perron-Frobenius theory~\cite{LAA,agn}.
An advantage of the present bound is that it is invariant
by scaling. For instance, with a state dependent
discount factor $<1$, it leads to a tighter bound than
the one which may be derived from~\eqref{borne-hansen} or~\eqref{notre_borne}.

Finally, we derive (Corollary~\ref{bound-for-games}) a strongly polynomial bound for the subclass of mean payoff games
such that there is a distinguished state to which the mean return time
is bounded by a constant $K=1/(1-\lambda)$, for every choice
of policies.
This condition implies that each transition matrix associated 
to a pair of policies of both players has a unique recurrence class,
and that there is a state which is common to each of these classes.

The paper is organized as follows. We present
background materials on zero-sum two-player stochastic games in Section~\ref{discrete},
and on policy iterations in Section~\ref{sec-policy}.
We state our results in Section~\ref{sec-main}.
Proofs or sketches of proofs, as well as some results 
of Perron-Frobenius theory on which they are based,
can be found in the following sections. 
\section{Two player zero-sum stochastic games with discrete time and mean payoff}
\label{discrete}

\subsection{The game processes}
Two player zero-sum stochastic games were introduced
by Shapley in the early fifties, see~\cite{Shapley}. We recall in this
section basic definitions in the case of finite state
space and discrete time (for more details see 
\cite{Shapley,FilarVrieze}).
When there is only one player (the set of actions of one of the 
two players is reduced to a singleton), such a game is
more commonly called a Markov Decision Process (MDP),
we refer to~\cite{Howard60,Denardo-Fox68,puterman}
for this topic.

We consider the finite state space $\X := \{1, \dots, n\}$.  A
stochastic process $\proc{\Xk_k}$ on $\X$ gives the state of the game
at each point time $k$, called {\em stage}. At each of these stages, two 
players, called ``\MIN'' and  ``\MAX'' (the minimizer and the maximizer)
 have the possibility to influence the course of the game. 

The \new{stochastic game} $\Gamma(\sx_0)$ starting from $\sx_0 \in \X$ is played
in stages as follows. The initial state  $\Xk_0$ is equal to $\sx_0$ and known by
the players. Player \MIN\ plays first, and chooses an action
$\Ak_0$ in a set of possible actions $\A_{\Xk_0}$. Then the second player,
\MAX, chooses an action $\Bk_0$ in a set of possible actions
$\B_{\Xk_0}$. The actions of both players and the current state
determine the payment $r_{\Xk_0}^{\Ak_0\Bk_0}$ made by \MIN\ to \MAX\ and the
probability distribution $\sy\mapsto P_{\Xk_0\sy}^{\Ak_0\Bk_0}$ of the new state
$\Xk_1$. Then the game continues from state $\Xk_1$, and so
on.

At a stage $k$, each player chooses an action knowing the \new{history}
defined by $\Ik_k = (\Xk_0,\allowbreak \Ak_0,\allowbreak \Bk_0, \allowbreak \cdots,\allowbreak \Xk_{k-1}, \allowbreak\Ak_{k-1},\allowbreak \Bk_{k-1},
\Xk_k)$ for \MIN\ and $(\Ik_k, \Ak_k)$ for \MAX.  
We call a \new{strategy} for a player, a rule which tells him the
action to choose in any situation. 
Assume $\A_\sx \subset \Ag$ and $\B_\sx \subset \Bg$ for some sets $\Ag$ and $\Bg$.
We shall consider only \new{pure Markovian strategies} for \MIN\ (resp. \MAX).
The latter are sequences $\salpha := (\balpha_0, \balpha_1, \cdots)$
(resp.\ $\sbeta := (\bbeta_0, \bbeta_1, \cdots)$)  
where $\balpha_k$ is a map $\X \rightarrow \Ag$ such that $\balpha_k (\sx)
\in \A_{\sx}$ for all $\sx\in\X$ (resp.\ $\bbeta_k$ is a map 
$\X\times\Ag \rightarrow \Bg$ such that $\bbeta_k(\sx,a) \in \B_{\sx}
  \, \forall \sx \in \X, \, a\in\A_{\sx}$).
They are said to be \new{stationary} if they are independent of $k$.
Then $\balpha_k$ is also denoted by $\balpha$ and $\bbeta_k$ by $\bbeta$. 
Also $\salpha$ and $\sbeta$ are identified with $\balpha$ and
$\bbeta$ respectively. A pure Markovian stationary strategy is also called a \new{feedback policy} or simply a \new{policy}.

A strategy $\salpha=\proc{\balpha_k}$ (resp. $\sbeta=\proc{\bbeta_k}$) together with an initial state determines stochastic processes $\proc{\Ak_k}$ for the actions of \MIN, $\proc{\Bk_k}$ for the actions of \MAX\ and $\proc{\Xk_k}$ for the states of the game.
For instance, for each pair of feedback policies ($\balpha$, $\bbeta$) of the two players, the state process 
$\proc{\Xk_k}$ is a Markov chain on $\X$ with transition probability 
\[
P(\Xk_{k+1}=\sy \, | \,  \Xk_k = \sx)\,=\, P^{\balpha(\sx)\bbeta(\sx,\balpha(\sx))}_{\sx \sy}
 \quad \text{ for } \sx, \sy \in \X\enspace,
\]
and  $\Ak_k = \balpha(\Xk_k)$ and $\Bk_k = \bbeta(\Xk_k,\Ak_k)$.

\subsection{Non-uniformly discounted and mean payoff games}

The payoff of the game $\Gamma(\sx)$ starting from $\sx$ is the
expected sum of the rewards at all steps of the game 
that \MAX\ wants to maximize and \MIN\ to
minimize. In this paper we shall consider 
games with an infinite horizon and a
discount factor $\gamma$, which is not uniform in that
it depends both on the state and
actions, $\gamma: \X\times \A\times \B\to [0,\infty)$.
We allow $\gamma(i,a,b)$ to take values larger that $1$
for some $(i,a,b)$.
The reward at time $k$ is defined to be the payment
made by \MIN\ to \MAX\ multiplied by all discount factors 
from time $0$ to time $k-1$. Thus, when the strategies $\salpha$ for
\MAX\ and $\sbeta$ for \MIN\ are fixed, the infinite horizon
discounted payoff of the game
$\Gamma(\sx, \salpha, \sbeta)$ starting from $\sx$ is given by
\[
J^{\gamma}(\sx, \salpha, \sbeta)\,=\, \sE^{\salpha \sbeta}_{\sx} \left[ \, \sum_{k = 0}^{\infty} \Big(\prod_{\ell=0}^{k-1}\gamma(\Xk_{\ell},\Ak_{\ell}, \Bk_{\ell})\Big)
r_{\Xk_k}^{\Ak_k \Bk_k}  \,\right], 
\]
where $\sE^{\salpha, \sbeta}_{\sx}$ denotes the expectation for the
probability law determined by the choice of strategies.
When  $\gamma\leq 1$, meaning that $\gamma(i,a,b)\leq 1 $ holds for all $i\in[n]$, $a\in A$, $b\in B$,
the above discounted game can be seen equivalently as a game which has,
at each stage, a stopping probability equal to $1-\gamma(\sx,a,b)$.

In all the paper, we shall assume, that 
\begin{assumption}\label{assump-finite}
the action spaces  $\A_\sx$ and $\B_\sx$, 
$\sx\in\X$,  are finite sets.
\end{assumption}
We shall write $\gamma\ll 1$ when
the discount factor is such that $\gamma(i,a,b) < 1 $ holds for all $i\in[n]$, $a\in A$, $b\in B$. This is the case 
if and only if there exists a scalar $\lambda$ such that:
\begin{assumption}\label{assump-disc}
$\gamma(i,a,b) \leq \lambda$, for all $i\in[n]$, $a\in A$, $b\in B$,
with $\lambda\in [0,1)$.
\end{assumption}
Then, one can transform the above discounted game into a game with 
an additional state (a ``cemetery'' state)
and a discount factor identically equal to $\lambda$
(independent of state and actions).
We can then apply to this situation earlier results 
concerning constant discount factors.

We shall also consider mean payoff games, defined as follows.
When the strategies $\salpha$ for \MIN\ and $\sbeta$ for \MAX\ are fixed,
the \new{(undiscounted) payoff in finite horizon} $\T$ of the game 
$\Gamma(\sx, \salpha, \sbeta)$ starting from $\sx$ is 
\[
J^\T(\sx, \salpha, \sbeta)\,=\, \sE^{\salpha \sbeta}_{\sx} \left[ \, \sum_{k = 0}^{\T-1} r_{\Xk_k}^{\Ak_k \Bk_k} \,\right], 
\]
and its \new{mean payoff} is
\[
J(\sx, \salpha, \sbeta)\,=\, \limsup_{\T \rightarrow \infty }\,
\frac{1}{\T} \, J^\T(\sx, \salpha, \sbeta).
\]

The discounted infinite horizon game with a discount factor
$\gamma\ll 1$, 
the finite horizon game and the mean payoff game, are all known
to have a \new{value},
denoted respectively by $v^{\gamma}_\sx$, $v^\T_\sx$ and $\meanv_\sx$, 
\begin{eqnarray} 
v^{\gamma}_\sx &:=& \inf_{\salpha} \sup_{\sbeta} J^{\gamma}(\sx, \salpha, \sbeta)
=  \sup_{\sbeta} \inf_{\salpha}\, J^{\gamma}(\sx, \salpha, \sbeta),
\label{valueGdisc}\\
v^\T_\sx &:=& \inf_{\salpha} \sup_{\sbeta} \, J^\T(\sx, \salpha, \sbeta)
=  \sup_{\sbeta} \inf_{\salpha}\, J^\T(\sx, \salpha, \sbeta),
\label{valueG-MP-finite}\\
\meanv_\sx &:=& \inf_{\salpha} \sup_{\sbeta} \,J(\sx, \salpha, \sbeta),
=  \sup_{\sbeta} \inf_{\salpha}\,J(\sx, \salpha, \sbeta), \label{valueG-MP}
\end{eqnarray} 
for all initial states $\sx \in \X$,
where the infimum is taken among all strategies $\salpha$ for \MIN\
and the supremum is taken over all strategies $\sbeta$ for \MAX\
(we refer the reader to~\cite{Shapley} for finite horizon or discounted infinite horizon games with constant discount factor, and to~\cite{liggettlippman} for mean payoff games).

Optimal strategies for both players (together
with the value of the game $\Gamma$ for every initial state)
can be obtained by the dynamic programming approach~\cite{Shapley},
which we next recall.

\subsection{Dynamic programming equations}
\label{sec-dynamic}

When considering finite horizon or mean-payoff games,
we assume that the discount factor $\gamma(i,a,b)$ 
at every state and node is identically 
equal to $1$, written $\gamma\equiv 1$.
To handle in the same setting the discounted and
the mean payoff cases, it will be convenient to consider
the following unnormalized nonnegative cooefficients,
rather than the transition probabilities:
\[ M_{\sx \sy}^{ab}=\gamma(\sx,a,b) P_{\sx \sy}^{ab}\quad \forall \sx,\sy\in\X,
a\in\A_\sx,\; b\in \B_\sx\enspace.\]
We wil also use the following notation,
for all $\sx \in \X,\; a\in\A_\sx,\; b\in \B_\sx$ and $v\in \RX$:
\begin{subequations}\label{formulaF}
\begin{eqnarray} 
\label{ILPIDE}
\fviab{F}{v}{\sx}{a}{b} & = &  \sum_{\sy \in \X}
M_{\sx \sy}^{ab}\, \valx_\sy \, + \, r_\sx^{a b} ;\\
\label{ILPI}
\fvia{F}{v}{\sx}{a} & = & \max_{b \in \B_\sx} \, \fviab{F}{v}{\sx}{a}{b};\\
\label{eq1}
\fvi{F}{v}{\sx} & = &  \, \min_{a \in \A_\sx } \, \fvia{F}{v}{\sx}{a}.
\end{eqnarray}
\end{subequations}

The  \new{dynamic programming} or \new{Shapley operator} associated to
all above games is the self-map $f$ of $\RX$ given by:
\begin{equation} \label{eq1opdef}
[f(v)]_{\sx}:=\fvi{F}{v}{\sx} , \qquad  \forall \sx \in \X,\; v\in \RX.
\end{equation}

The value $v^\T=(v^\T_i)_{i\in \X}$ of the finite horizon game satisfies the 
\new{dynamic programming equation}~\cite{Shapley} 
associated to the operator $f$:
\[ v^{\T+1}\, = \, f(v^\T),\qquad \T=0,1,\dots 
\] 
with initial condition $v^0\equiv 0$ ($v^0_\sx=0,\; \sx \in \X$).

Similarly, the value $v^\gamma=(v^\gamma_i)_{i\in [n]}$
 of the discounted infinite horizon game,
with a discount factor $\gamma\ll 1$, 
is the unique solution $v\in \RX$
of the (stationary) dynamic programming equation~\cite{Shapley}:
\begin{equation} \label{eq1muop}
v\, = \, f(v).
\end{equation}
Also, optimal strategies are obtained for both players by taking pure Markovian stationary strategies $\balpha$ for \MIN\ and $\bbeta$ for \MAX\ such that,
for all $\sx\in\X$, and $a\in\A_{\sx}$,
$\balpha(\sx)$ attains the minimum in the expression of
$\fvi{F}{v}{\sx}$ in~\eqref{eq1},
and $\bbeta(\sx,a)$ attains the maximum in the expression 
of $\fvia{F}{v}{\sx}{a}$ in~\eqref{ILPI}.

The dynamic programming operator $f$ is always \new{order-preserving}, 
i.e., $\valx\leq \valy\implies f(\valx)\leq f(\valy)$ where
$\leq$ denotes the partial ordering of $\RX$
($\valx\leq \valy$ if $\valx_{\sx}\leq \valy_{\sx}$ for all $i\in [n]$).
When $\gamma\leq 1$,  $f$ is also \new{additively subhomogeneous}, 
meaning that it satisfies
$f(\lambda+\valx)\leq \lambda+f(\valx)$ for all $\lambda\in\R$ 
nonnegative ($\lambda\geq 0$) and $\valx\in\RX$, where
$\lambda + \valx:=(\lambda+\valx_i)_{i\in[n]}$.
This implies that $f$ is nonexpansive in the sup-norm.
When in addition Assumption~\ref{assump-disc} holds, the map $f$ 
is contracting in the sup-norm with contraction 
factor $\lambda$, that is:
\[ \|f(v)-f(w)\|\leq \lambda \|v-w\|\enspace,\]
where $\|\cdot\|$ denotes the sup-norm of $\R^n$
($\|v\|=\max\set{|v_i|}{i\in\X}$).
Then, one can solve the fixed point equation~\eqref{eq1muop}
of $f$ by using the fixed point iterations, also called
\new{value iterations} in the optimal control or game context:
$v^{k+1}=f(v^k)$. They will converge geometrically towards the solution
$v$ with factor $\lambda$:
$\lim_{k\to\infty} \|v^k-v\|^{1/k}\leq \lambda$. 
However the complexity of this algorithm is known to
be only pseudo polynomial.
Indeed the number of necessary iterations will depend on the norm
of the solution, which depends itself on the modulus of the 
parameters.

When now $\gamma\equiv 1$, $f$ is
\new{additively homogeneous}, meaning that it commutes
with the addition of a constant vector, i.e., that 
$f(\lambda+\valx)=\lambda+f(\valx)$ for all $\lambda\in\R$ and
$\valx\in\RX$. Then, the mean payoff of the game
can be studied through the following additive eigenproblem
\begin{equation} \label{eq1ergop}
\eta+v\, = \, f(v) \enspace. 
\end{equation}
Here, the vector $v\in \R^n$ is called an \new{additive eigenvector} of $f$
associated to the \new{additive eigenvalue} $\eta\in\R$.
If such an additive eigenpair exists,
then, the value of the mean payoff game represented by $f$
is equal to $\eta$ for all initial states.
Optimal strategies are obtained 
in the same way as for the discounted infinite horizon problem.

\section{Policy iteration algorithm: presentation and preliminary properties}
\label{sec-policy}

\subsection{Assumptions and notations}\label{sec-not}
Assume first that $f$ is given by~\eqref{eq1opdef}, with $F$ as
in~\eqref{formulaF},  $M_{\sx \sy}^{ab}$ nonnegative scalars,
and $\A_\sx$ and $\B_\sx$ finite sets (Assumption~\ref{assump-finite}).
Then, the sets of feedback policies
$\PA:=\set{\balpha : \X \rightarrow \Ag}{\balpha (\sx) \in \A_{\sx}\, \forall \sx\in\X }$
 and $\PB:= \set{\bbeta : \X\times\Ag \rightarrow \Bg}{\bbeta(\sx,a) \in \B_{\sx}
  \, \forall \sx \in \X, \, a\in\A_{\sx}}$ are finite.
Now for every pair of policies $\balpha\in\PA$ and $\bbeta\in\PB$ 
of the first and second players, we define the following matrices and vectors:
\[ M^{(\balpha\bbeta)}=(M^{\balpha_i\bbeta_i}_{ij})_{ij=1,\ldots, n}, 
\quad \text{and}\;
r^{(\balpha\bbeta)}=(r^{\balpha_i\bbeta_i}_{i})_{i=1,\ldots, n}\enspace,\]
and respectively the affine and nonlinear maps 
$f^{(\balpha\bbeta)}$ and $f^{(\balpha)}$ from $\RX$ to itself which coordinates are
given, for all $v\in \RX$, by:
\begin{subequations}\label{rectangular}
\begin{eqnarray}
f_\sx^{(\balpha\bbeta)}(v)&=& \fviab{F}{v}{\sx}{\balpha_\sx}{\bbeta_\sx},\\
f_\sx^{(\balpha)}(v)&=&\fvia{F}{v}{\sx}{\balpha_\sx} .\label{Fmin} 
\end{eqnarray}
\end{subequations}

Then, we can write, for all $v\in\R^n$,
\begin{subequations}
\begin{eqnarray}
f^{(\balpha\bbeta)}(v)&=&M^{(\balpha\bbeta)} v+r^{(\balpha\bbeta)},\\
f^{(\balpha)}(v)&=&\max_{\bbeta \in \PB}  f^{(\balpha\bbeta)}(v),\label{fmax}\\
f(v)&=&  \min_{\balpha\in \PA} f^{(\balpha)}(v)\enspace,\label{fmin}
\end{eqnarray}
\end{subequations}
where in these expressions, the maximum and the minimum
mean the supremum and infimum 
with respect to the partial order of $\R^n$. Note that it is attained 
for an element of $\PB$ and $\PA$ respectively.
Indeed, from~\eqref{rectangular}, the $i$-th entry of
$f^{(\balpha\bbeta)}$ and $f^{(\balpha)}$  
depends only on the policy at state $i$.
We shall say that a set of vectors $V\subset \mathbb{R}^n$
is \new{rectangular} if $V = \pi_1(V)\times \cdots \times \pi_n(V)$,
where $\pi_i: \R^n\to \R$ denotes the projection on the $i$th
coordinates.
It follows that 
the set of vectors $\{f^{(\balpha)}(v)\mid \balpha \in \PA\}$
is rectangular, and that for each $\balpha$, the set $\{f^{(\balpha\bbeta)}(v)\mid \bbeta\in\PB\}$ is also rectangular.

The maps  $f^{(\balpha\bbeta)}$ and $f^{(\balpha)}$ satisfy the same properties 
as the ones stated in Section~\ref{sec-dynamic} for $f$.
They are all order preserving.
When $\gamma\leq 1$ (resp.\ $\gamma\equiv 1$), they are additively 
subhomogeneous (resp.\ homogeneous), hence
nonexpansive in the sup-norm.
When Assumption~\ref{assump-disc} holds, these maps
are contracting in the sup-norm with contraction factor $\lambda$.

\subsection{Policy iteration algorithm for discounted games}
Here we are interested in solving Equation~\eqref{eq1muop}
by using the policy iteration algorithm for discounted games,
introduced by Howard~\cite{Howard60} for 1-player games, 
and by Hoffman and Karp~\cite{HoffmanKarp},
and Denardo~\cite{Denardo} for 2-player games.
It will be convenient to consider the following general algorithm.

\begin{algo}[General policy iteration algorithm]\label{policybase}
\ 

\noindent \emph{Input}: A set $\PA$, and maps $f$ and $f^{(\balpha)}$,
from $\RX$ to itself, for $\balpha\in\PA$, satisfying~\eqref{fmin},
for all $v\in\R^n$.

\noindent \emph{Output}: A fixed point $v$ of $f$ 
and a policy $\balpha\in\PA$ such that $f(v)=f^{(\balpha)}(v)$.
\begin{enumerate}
\item {\em Initialization}:
Set $\ki=0$.
Select an arbitrary strategy $\balpha^0\in\PA$.
\item\label{value} Compute the fixed point $v^{\ki}$ of $f^{(\balpha^{\ki})}$.
\item \label{improve} Improve the policy: choose an optimal policy 
for  $v^{\ki}$, that is $\balpha^{\ki+1}\in\PA$  such that 
$f(v^{\ki})=f^{(\balpha^{\ki+1})}(v^{\ki})$, with 
$\balpha^{\ki+1}= \balpha^{\ki}$ as soon as this is possible.
\item \label{stop} If $\balpha^{\ki+1}= \balpha^{\ki}$,
then the algorithm stops and returns $v^{\ki}$ and $\balpha^{\ki}$.
Otherwise, increment $\ki$ by one and go to   Step~\ref{value}.
\end{enumerate}
\end{algo}

When $\PA$ is as in Section~\ref{sec-not},
and $f^{(\balpha)}$ satisfies also~\eqref{Fmin}, Step~\ref{improve}
is equivalent to
\[ \balpha^{\ki+1}_\sx \ \in \ \argmin{a \in \A_\sx } \, \fvia{F}{v^{\ki}}{\sx}{a},
\quad \sx\in\X,\]
and we can also choose $\balpha^{\ki+1}$ in a conservative way, that is 
such that, for all $\sx\in \X$, $\balpha^{\ki+1}_\sx=\balpha^{\ki}_\sx$ as
soon as this is possible.
Algorithm~\ref{policybase} can also be applied to a map
$f$ satisfying~\eqref{fmin} with the min operation replaced by the max
operation.

When $f$ is as in Section~\ref{sec-not},
the policy iteration algorithm for 2-player games
consists in two levels of 
nested instances of the previous algorithm.

\begin{algo}[Policy iteration algorithm for 2-player games]\label{policy2}
\ 

\noindent \emph{Input}: A map $f$ given as in Section~\ref{sec-not}.

\noindent \emph{Output}: The value $v$ of the game associated to
$f$ and an optimal policy $\balpha\in\PA$.

\begin{itemize}
\item Apply Algorithm~\ref{policybase} (that is
construct the sequences $\balpha^\ki$ of policies and
$v^\ki$ of values, $\ki\geq 0$).
\item The solution $v^\ki$ in Step~\ref{value}
is the value of the game with fixed policy $\balpha^\ki$.
It is computed as follows:
\begin{itemize}
\item Apply  Algorithm~\ref{policybase} to the
set $\PB$ instead of $\PA$, the map $f^{(\balpha_\ki)}$
instead of $f$ and the maps $f^{(\balpha^{\ki}\bbeta)}$ with $\bbeta\in \PB$
instead of the maps $f^{(\balpha)}$ with $\balpha\in\PA$.
This constructs sequences of policies $\bbeta^{\ki,\kj}$ and values
$v^{\ki,\kj}$, with $\kj\geq 0$.
\item When the latter algorithm stops,
put $v^{\ki}=v^{\ki,\kj}$.
Then  $\bbeta^{\ki,\kj}$ is an optimal policy of the second player
of the game with fixed policy $\balpha^\ki$ for the first player.
\end{itemize}
\item When the algorithm stops, return $v^\ki$, $\balpha^\ki$
and $\bbeta^{\ki,\kj}$ with $\ki$ equal to the final index of
the external iteration of Algorithm~\ref{policybase}, 
and $\kj$ the final index of the internal iteration
of Algorithm~\ref{policybase}.
\end{itemize}
\end{algo}

Note that in the nested application of Algorithm~\ref{policybase},
Step~\ref{value} consists in solving a linear system,
which can be done either by a direct linear solver,
or approximately, by an iterative method.
In the present paper, we require an exact
solution.

The usual assumption for the validity of the above algorithms
is Assumption~\ref{assump-disc}.
Under this assumption, one can show (see
for instance~\cite{Bertsekas87}
for one-player games and~\cite{Denardo} for 2-player games)
that the sequence of values $(v^{\ki})_{\ki\geq 0}$ (resp.\ $(v^{\ki,\kj})_{\kj\geq 0}$
for some fixed $\ki\geq 0$) of Algorithm~\ref{policy2} 
is nonincreasing (resp.\ nondecreasing) and 
converges towards the unique fixed point
$v$ of $f$ (resp.\ $v^{\ki}$ of $f^{(\balpha^{\ki})}$),
and deduce in particular that Algorithm~\ref{policy2} (resp.\ each nested 
application of Algorithm~\ref{policybase}
in Algorithm~\ref{policy2})
never visits twice the same policy of the first (resp.\ second) player
(except before stopping).
Then, since the action spaces are finite, 
the algorithm (resp.\ nested policy iterations)
stops after a finite time.

These properties can indeed be obtained 
from the following result concerning
Algorithm~\ref{policybase}.

\begin{prop}\label{policybase-increasing-convergence}
Let $\PA$, $f$ and $f^{(\balpha)}$ be as in Algorithm~\ref{policybase}.
Assume that $\PA$ is finite, 
that the maps $f^{(\balpha)}$ are order preserving and
contracting in the 
sup-norm with the same contraction factor $\lambda$.
We have:
\begin{enumerate}
\item \label{convergence-1} $f$ is also order preserving and
contracting in the sup-norm with contraction factor $\lambda$;
\item \label{convergence-2}
the iterations of  Algorithm~\ref{policybase} are well defined;
\item \label{convergence-3} the sequence $v^\ki$ is nonincreasing and 
converges towards the unique fixed point $v$ of $f$;
\item \label{policyvalue}
more precisely: $v\leq v^{\ki+1}\leq f(v^{\ki})\leq v^\ki$;
\item \label{convergence-5} 
the sequence $\balpha^\ki$ never visits twice the same policy
(except when the stopping condition is satisfied);
\item \label{convergence-6} 
hence Algorithm~\ref{policybase} stops after a finite time.
\end{enumerate}
\end{prop}

From Point~\ref{policyvalue}, we see that the sequence $(v^{\ki})_{\ki\geq 0}$
of policy iteration algorithm~\ref{policy2} 
converges faster towards $v$ 
than the value iteration algorithm starting from $v^0$.
One can also deduce the following contraction property
(see for instance~\cite{hansen2011strategy}).

\begin{coro}\label{policybase-contraction}
Under the assumptions of Proposition~\ref{policybase-increasing-convergence},
the sequence $v^\ki$ 
satisfies the following contraction property in the sup-norm:
$\|v^{\ki+1}-v\|\leq \lambda  \|v^{\ki}-v\|$.
\end{coro}

\subsection{Policy iteration algorithm for mean-payoff games}
Now, we assume that $\gamma\equiv 1$, and are interested in solving
the optimality equation of the mean payoff problem, Equation~\eqref{eq1ergop},
by policy iteration.
The first algorithm doing so was introduced by Hoffman 
and Karp~\cite{HoffmanKarp}, assuming
all the matrices $M^{(\balpha\bbeta)}$ to be irreducible. 
This algorithm is very similar to the algorithm for discounted games,
so we only present here the differences. 

\begin{algo}[General policy iteration algorithm for the mean payoff additive eigenproblem]
\label{policymeanbase}
\ 

Same as Algorithm~\ref{policybase}, except that
\begin{itemize}
\item  $f$ and $f^{(\balpha)}$
are assumed to be additively homogeneous;
\item the initialization includes the
selection of an arbitrary state $\rs\in \X$;
\item instead of a  fixed point $v$ of $f$,
the algorithm is returning an additive eigenvector $v$ 
and associated eigenvalue $\eta$ of $f$ ($\eta+v=f(v)$)
such that $v_\rs=0$;
\item 
the computation of a fixed point  $v^{\ki}$ of
$f^{(\balpha^{\ki})}$ is replaced by the computation
of an additive eigenvector $v^{\ki}$ and associated eigenvalue $\eta^{\ki}$
of $f^{(\balpha^{\ki})}$ ($\eta^\ki+v^\ki=f^{(\balpha^{\ki})}(v^\ki)$),
such that $v^\ki_\rs=0$.
\end{itemize}
\end{algo}

Note that since the maps $f$ and  $f^{(\balpha)}$ are additively 
homogeneous, changing $\rs$ into another state $\tilde{\rs}$
does not change the admissible sequences of policies,
and only modifies the additive eigenvectors by additive constants.
Indeed, for any given input,
$\balpha^{\ki}$, $\eta^\ki$ and $v^{\ki}$ are 
respectively admissible sequences of policies,
additive eigenvalues, and additive eigenvectors
with $\rs$, if and only if
$\balpha^{\ki}$, $\eta^\ki$ and $\tilde{v}^{\ki}=v^\ki-v^\ki_{\tilde{\rs}}$
are respectively admissible sequences of policies,
additive eigenvalues, and additive eigenvectors
with $\tilde{\rs}$.

\begin{algo}[Hoffman and Karp policy iteration algorithm for 
2-player mean-payoff games]\label{policy2mean}
\ 

Same as Algorithm~\ref{policy2}, except that
\begin{itemize}
\item  we assume that $\gamma\equiv 1$;
\item instead of the value $v$ of the game associated to
$f$, the algorithm is returning the value $\eta$ and
a bias $v$ such that $v_\rs=0$;
\item Algorithm~\ref{policybase} is replaced by Algorithm~\ref{policymeanbase};
\item the algorithm constructs
the sequences of values $\eta^\ki$ and bias $v^\ki$ of
the game with a fixed policy $\balpha^\ki$, 
with $v^\ki_\rs=0$, and for each $\ki\geq 0$, 
it constructs the sequences of values $\eta^{\ki,\kj}$ and bias
$v^{\ki,\kj}$, $\kj\geq 0$, of the game with fixed policies $\balpha^\ki$ and 
$\bbeta^{\ki,\kj}$ of the first and second player,
with $v^{\ki,\kj}_\rs=0$.
\end{itemize}
\end{algo}

Some or all of the following properties can be found 
in~\cite{Bertsekas87}
for one-player games and~\cite{HoffmanKarp} for 2-player games.

\begin{prop}\label{policy-increasing-mean}
Assume all the matrices $M^{(\balpha\bbeta)}$, $\balpha\in\PA$, $\bbeta\in\PB$, 
are irreducible. Then, 
the sequence of values $(\eta^{\ki})_{\ki\geq 0}$ 
(resp.\ $(\eta^{\ki,\kj})_{\kj\geq 0}$
for some fixed $\ki\geq 0$)
of Algorithm~\ref{policy2mean} 
is nonincreasing (resp.\ nondecreasing)
and converges towards the unique eigenvalue 
$\eta$ of $f$ (resp.\ $\eta^{\ki}$ of $f^{(\balpha^{\ki})}$).
Also the sequence of bias $(v^{\ki})_{\ki\geq 0}$ 
(resp.\ $(v^{\ki,\kj})_{\kj\geq 0}$
for some fixed $\ki\geq 0$)
of Algorithm~\ref{policy2mean} 
converges towards the unique bias
$v$ of $f$ such that $v_\rs=0$ (resp.\ $v^{\ki}$ of $f^{(\balpha^{\ki})}$
such that $v^{\ki}_\rs=0$).
\end{prop}
\begin{coro}\label{policy-convergence-mean}
Assume all the matrices $M^{(\balpha\bbeta)}$, $\balpha\in\PA$, $\bbeta\in\PB$, 
are irreducible. Then, 
Algorithm~\ref{policy2mean} (resp.\ each nested 
application of Algorithm~\ref{policymeanbase}
in Algorithm~\ref{policy2mean})
never visits twice the same policy of the first (resp.\ second) player
(except when the stopping condition is verified).
Hence, the policy iterations (resp.\ nested policy iterations)
stop after a finite time.
\end{coro}

Note that the above algorithms cannot be applied
to multichain games (such that some matrices
$M^{(\balpha\bbeta)}$ have at least two final classes),
since then, the value of the game is not any more
given by a constant $\eta$ independent of the initial
state. See~\cite{CochetGaub,detournay2} for a discussion
of the multichain case.

\section{Bounds on the number of policy iterations}
\label{sec-main}

In the sequel, we shall state as far as possible
our results in the framework of the general policy iteration 
algorithms~\ref{policybase} and~\ref{policymeanbase}, the
application of these results to the zero-sum two-player game policy iteration 
algorithms~\ref{policy2} and~\ref{policy2mean} being immediate.

\subsection{Revisiting the bound of Ye and 
Hansen, Miltersen and Zwick with non linear maps}

The following improvement of~\cite{hansen2011strategy} is obtained
by the same arguments as in~\cite{hansen2011strategy}, 
except that we use the nonlinear maps
$f^{(\balpha)}$ directly instead of the affine maps $f^{(\balpha\bbeta)}$,
and that we use only sup-norms, whereas $\ell^1$ norms were used in some 
places in~\cite{hansen2011strategy}.

\begin{theorem}\label{hansen-improved}
Let $\PA$, $f$, $f^{(\balpha)}$ and $\lambda$
be as in Proposition~\ref{policybase-increasing-convergence}.
Assume also that $\PA$ is as in Section~\ref{sec-not},
and that $f^{(\balpha)}$ satisfies~\eqref{Fmin}.

Then the policy iteration  algorithm~\ref{policybase} stops after at most
$\ki_{\max}$ iterations, where $\ki_{\max}:= (m_1-n)  (1+\lfloor \frac{\log(1-\lambda)}{\log(\lambda)} \rfloor )$
and $m_1$ is the cardinality of 
$\SA:=\set{(\sx,a)}{\sx\in \X,\; a\in A_\sx }$.
\end{theorem}

\subsection{Discounted games with state dependent discount factors}
\label{sec-discount}
We denote by $\eigenrad{M}$ the spectral radius
of a $n\times n$ matrix $M$,
that is the maximum of the moduli of its eigenvalues.
When $\varphi\in\RX$ has strictly positive coordinates and
$v\in \RX$, we set
$\varphi^{-1}:=(\varphi_\sx^{-1})_{\sx \in \X}\in \RX$
and $\varphi v:= (\varphi_\sx v_\sx)_{\sx \in \X}\in \RX$
(these are the usual notations, if we identify
$\RX$ to the set of functions from $\X$ to $\R$).
For all self-maps $f$ of $\RX$, we denote by
$\scaled{\varphi}{f}$ its scaling by $\varphi$, which is
the map $\scaled{\varphi}{f}: v\mapsto  \varphi^{-1}f(\varphi v)$.
It is easy to see that if $f$ is order preserving so is $\scaled{\varphi}{f}$.a
The following result shows that these scalings leave invariant
the sequences of policies generated by the policy iteration algorithm. 
A sequence of policies and fixed points will be said to be {\em admissible}
for a given input if there is a valid run of the algorithm on this
input producing this sequence.
\begin{prop}[Scaling Invariance]\label{prop-general-scaling}
Let $\PA$, $f$ and $f^{(\balpha)}$ be as in Algorithm~\ref{policybase},
and let $\varphi\in\RX$ have strictly positive coordinates.
Denote $\tilde{f}:=\scaled{\varphi}{f}$
and $\tilde{f}^{(\balpha)}=\scaled{\varphi}{f^{(\balpha)}}$.
Then, $\PA$, $\tilde{f}$ and $\tilde{f}^{(\balpha)}$ constitute
a valid input of Algorithm~\ref{policybase}. Moreover,
$\balpha^{\ki}$ and $\tilde{v}^{\ki}=\varphi^{-1} v^\ki$
constitute an admissible sequence of policies
and fixed points for this input, if and only if
$\balpha^{\ki}$ and $v^\ki$ constitute an admissible sequence of policies
and fixed points for the original input $\PA$, $f$ and $f^{(\balpha)}$.
\end{prop}

For a set $\MC$ of $n\times n$ matrices
$\MC\subset \R^{n\times n}$, we shall define 
its \new{rectangular hull}, denoted $\rec{\MC}$,
as the set of matrices $N$ such that, for all $i\in\X$,
 the row $i$ of $N$ coincides with the row $i$ of some element $M$
of $\MC$.
When $g$ is a polyhedral self-map of $\RX$, so that
$\RX$ can be covered by finitely many polyhedra on which $g$ 
is affine, we shall denote by $\rgrad{g}$ the finite set of
matrices representing the differential of $g$ in each of these polyhedra.

The proof of the following result is based on nonlinear 
Perron-Frobenius theory and in particular on some results 
in~\cite{LAA,agn}.

\begin{theorem}\label{general-scaling}
Let $\PA$, $f$ and $f^{(\balpha)}$ be as in Algorithm~\ref{policybase}.
Assume that $\PA$ is as in Section~\ref{sec-not}, that
$f^{(\balpha)}$ satisfies~\eqref{Fmin},
and that the maps $f^{(\balpha)}$ are order preserving and
polyhedral.
Let $\MC(\balpha)= \rec{\rgrad{f^{(\balpha)}}}$ and
$\MC=\cup_{\balpha\in \PA} \MC(\balpha)$.
Assume that the spectral radii of all the matrices $M$
in $\MC$ are strictly less than $1$, and denote by 
$\omega$ the maximum of these spectral radii.
Then 
for all $\lambda$ such that
$\omega <\lambda<1$, there exists $\varphi\in\RX$ with strictly positive
coordinates such that the scaled maps $\tilde{f}:=\scaled{\varphi}{f}$
and $\tilde{f}^{(\balpha)}=\scaled{\varphi}{f^{(\balpha)}}$
are contracting in the sup-norm with contraction factor $\lambda$.
\end{theorem}

Using Theorem~\ref{general-scaling} and
Proposition~\ref{prop-general-scaling}, we obtain:
\begin{coro}
Under the assumptions of Theorem~\ref{general-scaling},
the conclusion of Proposition~\ref{policybase-increasing-convergence} holds.
\end{coro}

Applying Theorem~\ref{general-scaling}, 
Proposition~\ref{prop-general-scaling}, and
Theorem~\ref{hansen-improved} to all  $\lambda$ such that 
$\omega<\lambda<1$, we obtain:
\begin{coro}
Under the assumptions of Theorem~\ref{general-scaling},
the conclusion of Theorem~\ref{hansen-improved} holds
with $\lambda=\omega$.
\end{coro}

\begin{coro}\label{coro-omegabar}
Let $\PA$, $\PB$  and $f$ be  given as in Section~\ref{sec-not}.
Assume that the spectral radii of all the matrices $M^{(\balpha\bbeta)}$,
$\balpha\in\PA$, $\bbeta\in\PB$, are
strictly less than $1$, so that 
\(\bar{\omega}:=\max_{\balpha\in\PA,\bbeta\in\PB}
\eigenrad{M^{(\balpha\bbeta)}}<1\).
Then the conclusion of Theorem~\ref{hansen-improved} holds
for the policy iteration algorithm
for 2-player games, Algorithm~\ref{policy2}
(instead of Algorithm~\ref{policybase}),
with $\lambda=\bar{\omega}$.
\end{coro}

\begin{coro}
Let $\lambda\in [0,1)$ be fixed. Then,
the policy iteration algorithm solves in strongly polynomial
time the instances of zero-sum 2-player ``discounted''
stochastic games with perfect information
and state dependent discount factors (possibly locally greater than $1$)
that are such that the spectral radii of the transition matrices associated 
to every pair of policies of the two players is bounded
by $\lambda$.
\end{coro}

\subsection{Mean-payoff games with a renewal state}
\label{sec-mean-payoff}

For a Markov matrix $M$ and states $i,j$, we shall denote:
\[ \TC_{ij}(M)=\sE [\inf\set{k\geq 1}{X_k=j}\mid X_0=i]\enspace, \]
the expected first mean return time to state $j$ of a Markov chain
$X_k$ with transition matrix $M$ and initial state $i$.
It is easy to see that $\TC_{i\rs}(M)<+\infty$ for
all $i\in X$ if and only if $M$ has a unique final (recurrent) class
and that $\rs$ belongs to this class.
The state $\rs$ is called a \new{renewal state}.

The following transformation will allow us
to replace a self-map $f$ of $\RX$ 
by a sup-norm contraction. This will play a similar role
to the scaling transformation used in the discounted case.

Let $\varphi\in \RX$ have positive coordinates and $\rs\in \X$.
Then, the map $L_\varphi$ which to a couple $(\eta,v)$,
with $\eta\in \R$ and $v\in\RX$ such that $v_\rs=0$,
associates the vector $w=\eta+\varphi^{-1} v\in\RX$ is an affine
isomorphism, with inverse given by: $\eta=w_\rs$ and
$v= \varphi( w- w_\rs)$.
For all self-maps $f$ of $\RX$,
we shall denote by $\linear{\varphi}{f}$ the self-map of $\RX$,
such that for all $w,v\in\RX$ and $\eta\in\R$ with $v_\rs=0$
and $w=\eta+\varphi^{-1} v$,
we have $\linear{\varphi}{f}(w)=\varphi^{-1}(\eta (\varphi-1)+f(v))$.

\begin{prop}\label{prop-general-chang-mean}
Let $\PA$, $f$, and $f^{(\balpha)}$ be as in 
Algorithm~\ref{policymeanbase},
and let $\varphi\in\RX$ have strictly positive coordinates and $\rs\in\X$.
Denote $\tilde{f}:=\linear{\varphi}{f}$
and $\tilde{f}^{(\balpha)}=\linear{\varphi}{f^{(\balpha)}}$.
Then, $\PA$, $\tilde{f}$ and $\tilde{f}^{(\balpha)}$ are
a valid input of Algorithm~\ref{policybase}. Moreover
$\balpha^{\ki}$ and  $\tilde{v}^\ki=\eta^\ki+\varphi^{-1} v^\ki$
constitute an admissible sequence of policies
and fixed points for Algorithm~\ref{policybase} on this input,
if and only if $\balpha^{\ki}$, $\eta^\ki$ and $v^\ki$
constitute an admissible sequence of policies,
additive eigenvalues and additive eigenvectors 
for Algorithm~\ref{policymeanbase} on the original
input $\PA$, $f$, $f^{(\balpha)}$, when $\rs$ is chosen.
\end{prop}

\begin{theorem}\label{general-chang-mean}
Let $\PA$, $f$, and $f^{(\balpha)}$ be as in 
Algorithm~\ref{policymeanbase}.
Assume that $\PA$ is as in Section~\ref{sec-not}, that
$f^{(\balpha)}$ satisfies~\eqref{Fmin}, and
that the maps $f^{(\balpha)}$ are order preserving and polyhedral.
Let $\MC(\balpha)= \rec{\rgrad{f^{(\balpha)}}}$ and
$\MC=\cup_{\balpha\in \PA} \MC(\balpha)$.
Then, all matrices $M$ in $\MC$ are Markov matrices.
Assume that they all have a unique final class, and there there is a 
state $\rs\in \X$ which is common to each of these classes, so that 
\[ \TC_{i\rs}:=\max_{M\in \MC} \TC_{i\rs}(M) <+\infty \quad \forall i \in \X\enspace. \]
Let $\varphi\in\RX$ be the vector with coordinates $\varphi_i =\TC_{i\rs}
\geq 1$, and $K=\max_{i\in \X} \TC_{i\rs}$.
Then, the transformed maps $\tilde{f}:=\linear{\varphi}{f}$
and $\tilde{f}^{(\balpha)}=\linear{\varphi}{f^{(\balpha)}}$
are order-preserving and
contracting in the sup-norm with contraction factor $\lambda=(K-1)/K$.
\end{theorem}

\begin{coro}
Under the assumptions of Theorem~\ref{general-chang-mean},
Assertions~\ref{convergence-2},\ref{convergence-3},\ref{convergence-5},
and \ref{convergence-6}
of Proposition~\ref{policybase-increasing-convergence} hold
for Algorithm~\ref{policymeanbase} instead of Algorithm~\ref{policybase},
with $v^\ki$ replaced by $\eta^\ki+\varphi^{-1} v^\ki$,
and $v$ replaced by $\eta+\varphi^{-1} v$.
\end{coro}

Applying Theorem~\ref{general-chang-mean}, 
Proposition~\ref{prop-general-chang-mean}, and 
Theorem~\ref{hansen-improved}, we obtain:
\begin{coro}\label{bound-for-games}
Under the assumptions of Theorem~\ref{general-chang-mean},
the policy iteration  algorithm~\ref{policymeanbase}
stops after at most
$\ki_{\max}$ iterations, where $\ki_{\max}:= (m_1-n)  (1+\lfloor \frac{\log(K)}{\log(K/(K-1))} \rfloor ) = {\mathcal O} ((m_1-n) K \log K )$,
$K=\max_{i\in\X}\TC_{i\rs}$,
and $m_1$ is the cardinality of $\SA$.
\end{coro}

\begin{coro}
Let $\PA$ and $f$ be  given as in Section~\ref{sec-not}, with $\gamma\equiv 1$.
Assume that every matrix $M^{(\balpha\bbeta)}$, $\balpha\in\PA$, $\bbeta\in\PB$, 
 has a unique final class, and there is a 
state $\rs\in \X$ which is common to each of these classes, so that 
\[ \bar{\TC}_{i\rs}:=\max_{\balpha\in\PA,\bbeta\in\PB}
\TC_{i\rs}(M^{(\balpha\bbeta)}) <+\infty \quad \forall i \in \X\enspace. \]
Then, the conclusion of Corollary~\ref{bound-for-games} holds
for the Hoffman and Karp policy iteration algorithm
for 2-player mean-payoff games, Algorithm~\ref{policy2mean}
(instead of Algorithm~\ref{policymeanbase})
with $K=\max_{i\in\X}\bar{\TC}_{i\rs}$.
\end{coro}

\begin{coro}
Let $K\in [1,+\infty)$ be fixed. Then,
the Hoffman and Karp policy iteration algorithm
solves in strongly polynomial time the instances of zero-sum 2-player
stochastic mean-payoff games with perfect information having
a distinguished state to which the mean return time
is bounded by $K$ for all choices of
policies of both players.
\end{coro}

\section{Proof of the preliminary results of Section~\ref{sec-policy}}
The following proof is similar to the proofs of
the same properties for Algorithm~\ref{policy2} that can be found
for instance in~\cite{Bertsekas87}.
\begin{proof}[Proof of Proposition~\ref{policybase-increasing-convergence}]
Let $\PA$, $f$ and $f^{(\balpha)}$ be as in the proposition.
From~\eqref{fmin}, $f$ is order preserving and
contracting in the sup-norm with contraction factor $\lambda$.

Hence the maps $f^{(\balpha)}$ and $f$ have a unique fixed point, 
which implies that
Step~\ref{value} of Algorithm~\ref{policybase} is well defined.

Let $(v^{\ki})_{\ki\geq 1}$ be the sequence of Algorithm~\ref{policybase}.
We have $v^{\ki}=f^{(\balpha^{\ki})}(v^{\ki})\geq f(v^{\ki})=f^{(\balpha^{\ki+1})}(v^{\ki})$.
In particular,  $v^{\ki}\geq f^{(\balpha^{\ki+1})}(v^{\ki})$, which implies that
the sequence $(f^{(\balpha^{\ki+1})})^k(v^{\ki})$ is nonincreasing.
By the fixed point theorem for the contracting map $f^{(\balpha^{\ki+1})}$, 
the former sequence converges towards the unique fixed point,
which by definition is $v^{\ki+1}$.
This implies in particular that  $v^{\ki}\geq f^{(\balpha^{\ki+1})}(v^{\ki})\geq 
v^{\ki+1}$, so that the sequence  $(v^{\ki})_{\ki\geq 1}$ is nonincreasing.

Moreover, from the above equations,
we deduce that $v^{\ki}\geq f(v^{\ki})\geq v^{\ki+1}$.
In particular, the sequence $f^k(v^{\ki})$  is nonincreasing.
Again, by the fixed point theorem for the contracting map $f$,
the former sequence converges towards the unique fixed point $v$ of $f$,
hence  $v^{\ki}\geq v$ for all $\ki$.
Since the sequence $(v^{\ki})_{\ki\geq 1}$ is nonincreasing
and lower bounded by $v$, it converges towards
some vector $w\geq v$.
Then, from the above equations, we also get that 
$v^{\ki}\geq f(v^{\ki})\geq v^{\ki+1} \geq v$,
for all $\ki$, passing to the limit and using the continuity of $f$,
we deduce that $w=f(w)$, and since $f$ has a unique fixed point, we deduce that
$w=v$.

Assume by contradiction that the sequence $\balpha^{\ki}$ 
visits twice the same policy. This means that
 $\balpha^{\ki'}=\balpha^{\ki}$ for some $\ki'>\ki\geq 0$.
Since the map $f^{(\balpha^{\ki})}=f^{(\balpha^{\ki'})}$ has a unique fixed point,
we get that $v^{\ki}=v^{\ki'}$.
Since we already proved that the sequence
$(v^{\ki})_{\ki\geq 0}$ is nonincreasing, we obtain
$v^{\ki}\geq v^{\ki+1}\geq v^{\ki'}$. This implies that $v^\ki=v^{\ki+1}$,
hence $v^\ki=f(v^\ki)$,
so that, by definition, the algorithm necessarily stops at iteration $\ki$
if it did not stopped before, hence the iteration $\ki'$ does not
occur, and $\balpha^{\ki'}$ is computed only if
$\ki'=\ki+1$ and so the algorithm cannot visits twice the same policy,
except when the stopping condition is verified.

This implies that Algorithm~\ref{policybase} stops
after at most a number of iterations equal to the cardinality of
the set $\PA$.
\end{proof}

\begin{proof}[Proof of Corollary~\ref{policybase-contraction}]
From $v^{\ki}\geq f(v^{\ki})\geq v^{\ki+1} \geq v$,
we get that $\|v^{\ki+1}-v\|\leq \|f(v^{\ki})-v\|$ and since
$f$ is contracting with factor $\lambda$, we deduce that
$\|v^{\ki+1}-v\|\leq \lambda \|v^{\ki}-v\|$.
\end{proof}

\section{Proof of Theorem~\ref{hansen-improved}}
Let $v$ denote the unique fixed point of $f$,
and for all $\balpha\in\PA$, denote by 
$R^{(\balpha)}= f^{(\balpha)}(v)-v$ the residual induced by $v$ on the
fixed point equation of $f^{(\balpha)}$.
By~\eqref{fmin}, we have $\min_{\balpha\in \PA} R^{(\balpha)}=0$,
so $R^{(\balpha)}\geq 0$ for all $\balpha\in\PA$.
Moreover, $\balpha$ is an optimal policy of the game if and
only if $R^{(\balpha)}=0$, or equivalently $\|R^{(\balpha)}\|=0$.
Finally, by~\eqref{Fmin},
 $R^{(\balpha)}_\sx= R^{\balpha_\sx}_\sx$ for all $\sx\in\X$, where
$R^a_\sx= \fvia{F}{v}{\sx}{a} -v_\sx$ for all $\sx\in\X$ and
$a\in \A_\sx$ plays the role of a new reward such that the 
value of the dynamic programming equation is identically equal to zero.

Let $v^\ki$ and $\balpha^\ki$ be the sequences of values and 
policies constructed in Algorithm~\ref{policybase}.
Since $v^\ki\geq v$, $v^\ki=f^{(\balpha^{\ki})}(v^\ki)$,
and $f^{(\balpha^{\ki})}$ is order preserving, we get that
$v^\ki\geq f^{(\balpha^{\ki})}(v)\geq f(v)=v$,
hence $0\leq R^{(\balpha^{\ki})} \leq  v^\ki-v$ and taking the supremum
over all coordinates (or states), we get that
$\|R^{(\balpha^{\ki})}\| \leq  \| v^\ki-v\|$.
Now, since $v^\ki$ is the unique fixed point of the $\lambda$-contracting
map $f^{(\balpha^{\ki})}$, we get that
$\|v^\ki-f^{(\balpha^{\ki})}(v)\|\leq \lambda \|v^\ki-v\|$.
Then, $\| v^\ki-v\|\leq \|v^\ki-f^{(\balpha^{\ki})}(v)\|+\|R^{(\balpha^{\ki})}\|
\leq \lambda \|v^\ki-v\|+\|R^{(\balpha^{\ki})}\|$.
From all the above inequalities, we get that
\[ \|R^{(\balpha^{\ki})}\| \leq  \| v^\ki-v\|\leq \frac{1}{1-\lambda }
\|R^{(\balpha^{\ki})}\| \enspace.\]
Combining these inequalities with the contraction
of policy iterations shown in 
Corollary~\ref{policybase-contraction}
($\|v^{\ki+1}-v\|\leq \lambda  \|v^{\ki}-v\|$), we obtain that
for all $t\geq \ki+ p$, 
\[ \|R^{(\balpha^{t})}\| \leq  \mu \|R^{(\balpha^{\ki})}\| ,
\quad \text{with}\quad\mu=\frac{1}{1-\lambda} \lambda^p \enspace.\]
Moreover when $p= 1+\lfloor \log(1-\lambda)/\log(\lambda)\rfloor$
(which is the least integer such that  $p>\log(1-\lambda)/\log(\lambda)$),
we have $\mu<1$.

For all $\balpha\in\PA$, let us denote by 
$\graph(\balpha)$ the graph of
$\balpha$: $\graph(\balpha)=\set{(\sx,\balpha_\sx)}{\sx\in \X }$.
Since $R^{(\balpha)}_\sx= R^{\balpha_\sx}_\sx$ for all $\sx\in\X$, we get that
$\|R^{(\balpha)}\|=\max_{(i,a)\in \graph(\balpha)} R_i^a$.
Assume $\balpha^{\ki}$ is not optimal, then $\|R^{(\balpha^{\ki})}\|>0$
and let $(i,a)$ realizes the maximum of $R_i^a$ on $\graph(\balpha^{\ki})$.
If $t\geq \ki+ p$, with $p$ as before, and $(i,a) \in \graph(\balpha^{t})$,
we get that $R_i^a\leq  \|R^{(\balpha^{t})}\| \leq  \mu \|R^{(\balpha^{\ki})}\| 
=\mu R_i^a$ with $\mu<1$ and $R_i^a>0$, which is impossible.
This shows that
$(i,a)\not\in \graph(\balpha^{t})$,
 hence $\graph(\balpha^{t})\subset \SA\setminus\{(i,a)\}$,
for all $t\geq  \ki+ p$.
Let us construct a sequence $\SA_\ki$ of subsets of $\SA$,
equal to the empty set for all $\ki<p$, and such that for all $\ki\geq p$,
 $\SA_{\ki}$ is the union of $\SA_{\ki-1}$ with 
the set of couples $(i,a)$ realizing the maximum of $R_i^a$ 
on $\graph(\balpha^{\ki-p})$.
We get that $\graph(\balpha^{t})\subset \SA\setminus \SA_{t}$,
for all $t\geq  0$ and that for all $\ki\geq p$, there 
exist $(i,a)\in \SA_{\ki}\setminus\SA_{\ki-p}$,
as long as Algorithm~\ref{policybase} did not stop,
so that the cardinality of $\SA_{\ki}$ increases at least by one 
after each group of $p$ iterations.
Hence, $\SA\setminus \SA_{p(m_1-n)}$  has at most $n$ elements, and since,
for all $t\geq p(m_1-n)$, $\graph(\balpha^{t})\subset \SA\setminus \SA_{p(m_1-n)}$
and $\graph(\balpha^{t})$ has exactly $n$ elements,
we deduce that, if the algorithm did not stop before iteration number $t$,
there is only one choice for $\graph(\balpha^{t})$ with $t\geq p(m_1-n)$,
hence $\balpha^{t}=\balpha^{t+1}$, 
and the algorithm stops at iteration number $t$.

\section{Spectral radius notions and the results of Section~\ref{sec-discount}}
\label{sec-techniques}

Let $C$ be a closed convex cone of $\R^n$,
let $\inte{C}$ denote its interior, and
let $h$ be a nonlinear continuous positively homogeneous map from $C$ to
itself ($h(\lambda v)=\lambda h(v)$ for all $\lambda>0$ and $v\in C$).
The following definitions are taken from~\cite{Nuss-Mallet}:
\begin{itemize}
\item $v$ is an eigenvector of $h$ in $C$, and $\lambda$ is an eigenvalue 
associated to $v$, if $h(v)=\lambda v$.
\item The \new{cone eigenvalue spectral radius} of $h$
 is the supremum of its eigenvalues in $C$:
\[ \eigenvalspr{C}(h):= \sup\set{\lambda\geq 0}{
\exists v\in C \backslash \{0\}\;\text{such that}\;
h(v)=\lambda v}\enspace .\]
\item The \new{Collatz-{W}ielandt number} of $h$ is defined as:
\[ \cw{C}(h):=\inf\set{\lambda>0}{\exists v\in \inte{C}\;\text{such that}\;
h(v)\leq \lambda v}\enspace .\]
\item The \new{Bonsall's spectral radius} of $h$ is defined as:
\[ \bonsall{C}(h):=\inf_{k\geq 1} \|h^k\|_C^{1/k},
\quad \text{with}
\quad \|h\|_C:=\sup_{x\in C,\; \|x\|=1} \|h(x)\| \enspace , \]
for any given norm $\|\cdot\|$ on $\R^n$.
\end{itemize}

The equality $ \eigenvalspr{\R_+}(h)=\cw{\R_+}(h)$
in the following result was established by Nussbaum~\cite[Theorem~3.1]{LAA}.
The last equality is done in~\cite{agn} in a more general
infinite dimensional context, together with the first one.

\begin{theorem}[\protect{\cite[Theorem~3.1]{LAA}}, and~\cite{agn}]
\label{theorem-nussbaum}
For a continuous, positively homogeneous,
order preserving selfmap $h$ of $C=\R_+^n$, all the above
spectral radius notions of $h$ coincide:
\[ \eigenvalspr{\R_+}(h)=\cw{\R_+}(h)= \bonsall{\R_+}(h) \enspace. \]
We denote by $\eigenrad{h}$ this constant.
\end{theorem}

The following result can be deduced easily from 
Theorem~\ref{theorem-nussbaum}.
It is also proved in an infinite dimensional context in~\cite{agn}.
\begin{prop}\label{prop-rhosup}
Let $\Pi$ be a set, and $h$ and $h_{\pi}$, $\pi\in \Pi$,  be continuous, 
positively homogeneous, order preserving selfmaps of $\R_+^n$.
Assume that for all $v\in\R_+^n$, 
$h(v)=\max_{\pi\in \Pi}  h_\pi(v)$, meaning that 
$h(v)\geq h_\pi(v)$ for all $\pi\in \Pi$, and that 
there exists $\pi\in \Pi$ such that  $h(v)=h_\pi(v)$.
Then
\[ \eigenrad{h}=\max_{\pi\in \Pi}  \eigenrad{h_\pi}\enspace. \]
\end{prop}

\begin{proof}[Proof of Theorem~\ref{general-scaling}]
Since $f$ is order preserving, so is $\scaled{\varphi}{f}$.
If $f$ is a polyhedral map such that
all the matrices $M\in\rgrad{f}$ 
satisfy $M\varphi\leq \lambda \varphi$,
then, all the matrices $M'\in\rgrad{\scaled{\varphi}{f}}$
satisfy $M'\un=\varphi^{-1}M \varphi\leq \lambda\un $,
where $\un$ is the vector with all coordinates equal to $1$.
Then, since $M'$ has also nonnegative coordinates, because
$\scaled{\varphi}{f}$ is order preserving, we get that $M'$ is
contracting in the sup-norm with contraction factor $\lambda$.
Then, using the polyhedral and continuity properties of $f$,
it is easy to see that $\scaled{\varphi}{f}$ is also 
contracting in the sup-norm with contraction factor $\lambda$.

Let us show the above property for all maps $f^{(\balpha)}$.
For this, consider the self-map $\bar{f}$ of $\RX$ 
given by:
\begin{equation}\label{defbarf}
 \bar{f}(v):=\sup_{M\in\MC} (M v )
\enspace,\end{equation}
where $\MC$ is as in the theorem.
Since all the matrices involved in the previous formula 
have nonnegative entries, the corresponding 
self-maps of $\RX$ are order-preserving.
Since $\PA$ is finite 
and the maps $f^{(\balpha)}$ are polyhedral, 
the set $\MC$ is finite.
Since in addition $\PA$ is as in Section~\ref{sec-not},
the set $\MC$ is the Cartesian product of the sets of its rows, hence
the supremum in~\eqref{defbarf} is a maximum.
Then, applying Proposition~\ref{prop-rhosup},
we get that 
\[ \eigenrad{ \bar{f}}=\max_{M\in\MC} \eigenrad{M}= \omega.\]
In particular the maximum is attained, hence $<1$.
Now, from Theorem~\ref{theorem-nussbaum}, we get that
$\omega=\eigenrad{ \bar{f}}=\cw{\R_+}( \bar{f})$, 
hence for all $\lambda>\omega$, there exists $\varphi\in\RX$ with
positive coefficients such that $\bar{f}(\varphi)\leq \lambda \varphi$.
This implies that all the matrices $M\in\rgrad{f^{(\balpha)}}$
satisfy $M\varphi\leq \lambda \varphi$,
which by the above arguments implies that  $\scaled{\varphi}{f^\balpha}$ is
contracting in the sup-norm with contraction factor $\lambda$.
\end{proof}

\begin{proof}[Proof of Proposition~\ref{prop-general-scaling}]
By definition of $\scaled{\varphi}{f}$, we have
$v=f(v)$ if and only if $w=\scaled{\varphi}{f}(w)$
for $w=\varphi^{-1} v$.
Moreover, the transformation of maps $f$,
$\scaled{\varphi}{f}(w)=\varphi^{-1}f(\varphi w)$
preserves the order on the maps $f$.
Hence if $v^\ki$ and $\balpha^\ki$ are respectively
sequences  of fixed points, and policies of
Algorithm~\ref{policybase} for $f$ and $f^{(\balpha)}$,
then $w^\ki=\varphi^{-1} v^\ki$ and $\balpha^\ki$ are respectively
sequences  of fixed points and policies of
Algorithm~\ref{policybase} for $\scaled{\varphi}{f}$ and 
$\scaled{\varphi}{f^{(\balpha)}}$.
\end{proof}

\section{The results of Section~\ref{sec-mean-payoff}}

\begin{lemma}
Let $M$ be a $n\times n$ Markov matrix with a unique final class,
and let $\rs\in \X$ belong to this final class.
Denote by $M_{(\rs)}$ the matrix obtained from $M$ by putting to
zero all entries in the $c$-th column.
Then, the vector $\varphi\in\RX$ with coordinates $\varphi_i =\TC_{i\rs}(M)$
satisfies $\varphi=1+M_{(\rs)} \varphi$.
\end{lemma}

\begin{lemma}\label{lemma-equiv}
Let $M$ be a $n\times n$ Markov matrix with a unique final class, and
let $\rs\in \X$ belong to this final class.
Consider a vector $\varphi\in \RX$ with positive coordinates such that
$\varphi\geq 1+M_{(\rs)} \varphi$, and let $K$ be a bound
on its coefficients, $K\geq \|\varphi\|$.
Construct the $n\times n$ matrix $M_{(\rs,\varphi)}$ by
replacing the $c$-th column of $M$ with the vector
$(1/\varphi_\rs) (\varphi-1 -M_{(\rs)} \varphi)$.
Then, $M_{(\rs,\varphi)}$  has nonnegative entries and satisfies
\begin{equation}\label{defmphi}
M_{(\rs,\varphi)}\varphi=\varphi-1\leq \lambda \varphi\enspace,
\end{equation}
with $\lambda=(K-1)/K$.
Moreover, for all $\eta\in\R$ and $v\in \RX$ such that $v_\rs=0$, 
we have
\begin{equation}\label{equiv-mean}
Mv+\eta (\varphi -1)=M_{(\rs,\varphi)} (v+\eta \varphi)
\enspace .
\end{equation}
\end{lemma}
\begin{coro}\label{coro-equiv}
Under the conditions of Lemma~\ref{lemma-equiv},
the map $f(v)=Mv$ is such that
$\linear{\varphi}{f}(w)=\varphi^{-1}M_{(\rs,\varphi)} (\varphi w)=M'w$,
for some matrix $M'$ with non negative entries
and row sums less or equal to $\lambda$.
Hence, $\linear{\varphi}{f}$ is order-preserving and contracting with 
contraction factor $\lambda$.
\end{coro}
\begin{proof} Indeed,
$\linear{\varphi}{f}(w)=\varphi^{-1}M_{(\rs,\varphi)} (\varphi w)$
follows from~\eqref{equiv-mean}.
Since by~\eqref{defmphi}, $M_{(\rs,\varphi)} \varphi \leq \varphi$,
we deduce that $M'\un\leq \lambda \un$ where $\un$ denotes the 
the vector with all coordinates equal to $1$.
\end{proof}

\begin{proof}[Proof of Theorem~\ref{general-chang-mean}]
If $f$ is a polyhedral map such that, 
all matrices $M\in\rgrad{f}$ 
satisfy the conditions of Lemma~\ref{lemma-equiv} (with the same
fixed $\varphi$ and $\rs$),
then by Corollary~\ref{coro-equiv}, all matrices 
$M'\in\rgrad{\linear{\varphi}{f}}$
satisfy the conclusions of Corollary~\ref{coro-equiv}.
This implies by the continuity of $f$ and $\linear{\varphi}{f}$,
that $\linear{\varphi}{f}$ is order-preserving and contracting with 
contraction factor $\lambda$.

Let us show the above property for all maps $f^{(\balpha)}$.
For this, consider the self-map $\bar{f}$ of $\RX$
given by:
\begin{equation}\label{defbarfnew}
 \bar{f}(v):=\max_{M\in\MC} (M_{(\rs)} v )
\enspace.\end{equation}
Note that it coincides with map of~\eqref{defbarf} on the set of
vectors $v$ such that $v_\sx=0$, but we shall apply it to all vectors.
Since all matrices involved in the previous formula have 
a unique final class and that this class contains $\rs$, we get that
they have all a spectral radius strictly less than $1$.
By the same arguments as in the previous section, the map $\bar{f}$ 
has a spectral radius strictly less than one, so is contracting 
for the sup-norm after a scaling by some vector $\psi$
(or equivalently is contracting the 
weighted sup-norm $\|v\|_{\psi}=\|v\psi^{-1}\|$).
In particular the equation $\varphi=1+\bar{f}(\varphi)$ has a unique solution
$\varphi$, 
and since the set of $\MC$ is rectangular,
this equation is the dynamic programming 
equation of an infinite horizon discounted 1-player game problem.
The interpretation of $\varphi$ as the value of this 1-player game problem
gives that $\varphi_{\sx}=\TC_{\sx\rs}$ for all $\sx\in \X$.
Since $\varphi=1+\bar{f}(\varphi)\geq 1+M_{(\rs)} \varphi$ 
for all $M\in  \MC$, and a fortiori
for all $M\in\rgrad{f^{(\balpha)}}$ and $\balpha\in\PA$,
which implies that $M$ satisfies the conditions of Lemma~\ref{lemma-equiv}
with $\varphi$ and $\rs$, we get by the above arguments 
that  the maps $\linear{\varphi}{f^\balpha}$ are
contracting in the sup-norm with contraction factor $\lambda$.
\end{proof}
\begin{proof}[Proof of Proposition~\ref{prop-general-chang-mean}]
By definition of $\linear{\varphi}{f}$, we have
$\eta+v=f(v)$ if and only if $w=\linear{\varphi}{f}(w)$
for $w=\eta+\varphi^{-1} v$.
Moreover, the transformation of maps $f$,
$\linear{\varphi}{f}(w)=\varphi^{-1}(\eta (\varphi-1)+f(v))$,
is preserving the order on the maps $f$.
Hence if $\eta^\ki$, $v^\ki$ and $\balpha^\ki$ are respectively
the sequences  of eigenvalues, eigenvectors, and policies of
Algorithm~\ref{policymeanbase} for $f$ and $f^{(\balpha)}$,
then $w^\ki=\eta^\ki+\varphi^{-1} v^\ki$ and $\balpha^\ki$ are respectively
the sequence  of fixed points and policies of
Algorithm~\ref{policybase} for $\linear{\varphi}{f}$ and 
$\linear{\varphi}{f^{(\balpha)}}$.
\end{proof}

\bibliographystyle{plainurl}

\bibliography{references}

\def\cprime{$'$}
\begin{thebibliography}{10}

\bibitem{detournay2}
M.~Akian, J.~Cochet-Terrasson, S.~Detournay, and S.~Gaubert.
\newblock Policy iteration algorithm for zero-sum multichain stochastic games
  with mean payoff and perfect information, 2012.
\newblock \href {http://arxiv.org/abs/1208.0446} {\path{arXiv:1208.0446}}.

\bibitem{agn}
M.~Akian, S.~Gaubert, and R.~Nussbaum.
\newblock A {C}ollatz-{W}ielandt characterization of the spectral radius of
  order-preserving homogeneous maps on cones.
\newblock 2011.
\newblock \href {http://arxiv.org/abs/1112.5968} {\path{arXiv:1112.5968}}.

\bibitem{andersson2009}
D.~Andersson.
\newblock Extending {F}riedmann's lower bound to the {H}offman-{K}arp
  algorithm.
\newblock {\em preprint, June}, 2009.

\bibitem{Bertsekas87}
D.~P. Bertsekas.
\newblock {\em Dynamic programming}.
\newblock Prentice Hall Inc., Englewood Cliffs, NJ, 1987.
\newblock Deterministic and stochastic models.

\bibitem{CochetGaub}
Jean Cochet-Terrasson and St{\'e}phane Gaubert.
\newblock A policy iteration algorithm for zero-sum stochastic games with mean
  payoff.
\newblock {\em C. R. Math. Acad. Sci. Paris}, 343(5):377--382, 2006.

\bibitem{Denardo-Fox68}
E.~V. Denardo and B.~L. Fox.
\newblock Multichain {M}arkov renewal programs.
\newblock {\em SIAM J. Appl. Math.}, 16:468--487, 1968.

\bibitem{Denardo}
Eric~V. Denardo.
\newblock Contraction mappings in the theory underlying dynamic programming.
\newblock {\em SIAM Review}, 9:165--177, 1967.

\bibitem{fearnley2}
John Fearnley.
\newblock Exponential lower bounds for policy iteration.
\newblock In {\em Automata, Languages and Programming}, pages 551--562, 2010.

\bibitem{FilarVrieze}
Jerzy Filar and Koos Vrieze.
\newblock {\em Competitive {M}arkov decision processes}.
\newblock Springer-Verlag, New York, 1997.

\bibitem{Friedmann}
Oliver Friedmann.
\newblock An exponential lower bound for the parity game strategy improvement
  algorithm as we know it.
\newblock In {\em LICS}, pages 145--156. IEEE Computer Society, 2009.

\bibitem{hansen2011strategy}
T.D. Hansen, P.B. Miltersen, and U.~Zwick.
\newblock Strategy iteration is strongly polynomial for 2-player turn-based
  stochastic games with a constant discount factor.
\newblock In {\em Innovations in Computer Science 2011}, pages 253--263.
  Tsinghua University Press, 2011.

\bibitem{HoffmanKarp}
A.~J. Hoffman and R.~M. Karp.
\newblock On nonterminating stochastic games.
\newblock {\em Management Science. Journal of the Institute of Management
  Science. Application and Theory Series}, 12:359--370, 1966.

\bibitem{Howard60}
Ronald~A. Howard.
\newblock {\em Dynamic programming and {M}arkov processes}.
\newblock The Technology Press of M.I.T., Cambridge, Mass., 1960.

\bibitem{liggettlippman}
T.~M. Liggett and S.~A. Lippman.
\newblock Stochastic games with perfect information and time average payoff.
\newblock {\em SIAM Rev.}, 11:604--607, 1969.

\bibitem{Nuss-Mallet}
J.~Mallet-Paret and Roger Nussbaum.
\newblock Eigenvalues for a class of homogeneous cone maps arising from
  max-plus operators.
\newblock {\em Discrete and Continuous Dynamical Systems}, 8(3):519--562, July
  2002.

\bibitem{LAA}
R.D. Nussbaum.
\newblock Convexity and log convexity for the spectral radius.
\newblock {\em Linear Algebra Appl.}, 73:59--122, 1986.

\bibitem{puterman}
M.~L. Puterman.
\newblock {\em Markov decision processes: discrete stochastic dynamic
  programming}.
\newblock Wiley Series in Probability and Mathematical Statistics: Applied
  Probability and Statistics. John Wiley \& Sons Inc., New York, 1994.

\bibitem{Shapley}
L.~S. Shapley.
\newblock Stochastic games.
\newblock In {\em Stochastic games and applications ({S}tony {B}rook, {NY},
  1999)}, volume 570 of {\em NATO Sci. Ser. C Math. Phys. Sci.}, pages 1--7.
  Kluwer Acad. Publ., Dordrecht, 2003.
\newblock Reprint of Proc. Nat. Acad. Sci. U.S.A. {{\bf{3}}9} (1953),
  1095--1100 [0061807].

\bibitem{ye_strong}
Y.~Ye.
\newblock The simplex and policy-iteration methods are strongly polynomial for
  the {M}arkov decision problem with a fixed discount rate.
\newblock {\em Math. Oper. Res.}, 36(4):593--603, 2011.

\end{thebibliography}

\end{document}